\documentclass[12pt]{amsart}
\usepackage{a4wide,enumerate,xcolor}
\usepackage{amsmath,graphicx}
\allowdisplaybreaks
\usepackage[utf8]{inputenc}
\usepackage[english,vietnamese]{babel}
\usepackage{cite} 

\let\pa\partial
\let\na\nabla

\newcommand{\N}{{\mathbb N}}
\newcommand{\R}{{\mathbb R}}
\newcommand{\diver}{\operatorname{div}}

\newtheorem{theorem}{Theorem}
\newtheorem{lemma}[theorem]{Lemma}

\newtheorem{remark}[theorem]{Remark}

\begin{document}

\title[Existence analysis of a coupled memristor--electric network system]{Existence analysis of a three-species memristor \\ drift-diffusion system coupled to electric networks} 

\author[A. J\"ungel]{Ansgar J\"ungel}
\address{Institute of Analysis and Scientific Computing, TU Wien, Wiedner Hauptstra\ss e 8--10, 1040 Wien, Austria}
\email{juengel@tuwien.ac.at} 

\author[\selectlanguage{vietnamese}T. T. Nguyễn]{\selectlanguage{vietnamese}Tuấn Tùng Nguyễn}
\address{Institute of Analysis and Scientific Computing, TU Wien, Wiedner Hauptstra\ss e 8--10, 1040 Wien, Austria}
\email{tuan-tung.nguyen@asc.tuwien.ac.at}

\selectlanguage{english}
\date{\today}

\thanks{The authors acknowledge partial support from   
the Austrian Science Fund (FWF), grant 10.55776/F65, and from the Austrian Federal Ministry for Women, Science and Research and implemented by \"OAD, project MultHeFlo. This work has received funding from the European Research Council (ERC) under the European Union's Horizon 2020 research and innovation programme, ERC Advanced Grant NEUROMORPH, no.~101018153. For open-access purposes, the authors have applied a CC BY public copyright license to any author-accepted manuscript version arising from this submission.} 

\begin{abstract}
The existence of global weak solutions to a partial-differential-algebraic system is proved. The system consists of the drift-diffusion equations for the electron, hole, and oxide vacancy densities in a memristor device, the Poisson equation for the electric potential, and the differential-algebraic equations for an electric network. The memristor device is modeled by a two-dimensional bounded domain, and mixed Dirichlet--Neumann boundary conditions for the electron and hole densities as well as the potential are imposed. The coupling is realized via the total current through the memristor terminal and the network node potentials at the terminals. The network equations are decomposed in a differential and an algebraic part. The existence proof is based on the Leray--Schauder fixed-point theorem, a priori estimates coming from the free energy inequality, and a logarithmic-type Gagliardo--Nirenberg inequality. It is shown, under suitable assumptions, that the solutions are bounded and strictly positive.
\end{abstract}

 

\maketitle


\section{Introduction}

Neuromorphic computing is a brain-inspired approach to developing hardware and algorithms that can make artificial neural networks more energy-efficient. A promising device as technology enabler of neuromorphic computing is the memristor. It is a two-terminal nonlinear resistor with memory showing a resistive switching behavior, which mimicks biological synaptic functions in artificial neural networks \cite{IeAm20}. Memristors are often part of an electric network, containing resistors, capacitors, and inductors; see Figure \ref{fig.device}. In this paper, the memristor device is modeled by three-species drift-diffusion equations, self-consistently coupled to the Poisson equation for the electric potential, and the electric network is described by differential-algebraic equations. In the literature, only two-species drift-diffusion equations have been considered with electric networks, either for steady states \cite{ABG10,ABGT03,ABR09,AlRo10} or transient states \cite{ABG05,AlRo21}. The extension to more than two species is nontrivial, since monotonicity properties of the drift terms cannot be used anymore. This issue is overcome by higher-order elliptic regularity, estimates from the free energy, and a logarithmic-type Gagliardo--Nirenberg inequality. The main aim of this paper is to prove the existence of global weak solutions in two space dimensions.

\begin{figure}[ht]
\includegraphics[width=80mm]{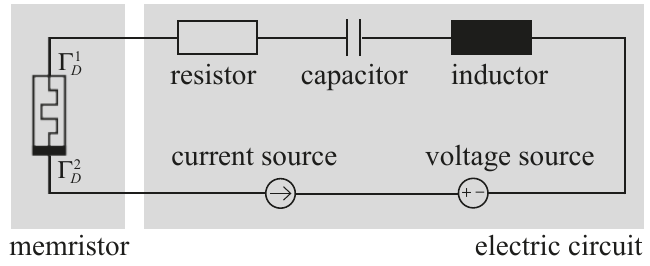}
\caption{Illustration of a simple coupled device--circuit model, consisting of a memristor and an electric circuit with resistor, capacitor, and inductor. The terminals of the memristor are denoted by $\Gamma_D^1$ and $\Gamma_D^2$.}
\label{fig.device}
\end{figure}

\subsection{Model equations}

The dynamics in semiconductor devices is usually governed by the electrons in the conduction band and by the holes (electron vacancies) in the valence band of the semiconductor lattice. In metal--oxide memristors, the oxide vacancies (missing oxygen atoms) play a crucial role. The movement of oxygen vacancies under an electric field enables the memristor to switch between high and low resistance states, which is the basis for the memristor functionality. Therefore, the charge carriers are negatively charged electrons, positively charged holes, and positively charged oxide vacancies. 

We model the memristor device by the domain $\Omega\subset\R^2$ and assume that the carrier flow can be described by the (scaled) drift--diffusion equations for the electron density $n$, hole density $p$, density of the oxide vacancies $D$, and the electric potential $V$ \cite{GSTD13,SBW09}:
\begin{align}
  & \pa_t n = \diver J_n, \quad \pa_t p = -\diver J_p, \quad
  \pa_t D = -\diver J_D, \label{1.dd1} \\
  & J_n = \na n-n\na V, \quad J_p = -(\na p+p\na V), \quad
  J_D = -(\na D+D\na D), \label{1.dd2} \\
  & \lambda^2\Delta V = n-p-D+A(x)\quad\mbox{in }\Omega,\ t>0,
  \label{1.V}
\end{align}
where $J_n$, $J_p$, and $J_D$ are the current densities of the electrons, holes, and oxide vacancies, respectively. The boundary of the memristor device consists of two terminals $\Gamma_D^1$ and $\Gamma_D^2$ with $\Gamma_D^1\cup\Gamma_D^2=\Gamma_D$ and the insulating boundary parts $\Gamma_N$ such that $\pa\Omega = \overline{\Gamma}_D\cup \overline{\Gamma}_N$. It is straightforward to generalize the model to more than two terminals.  Following \cite{SBW09}, we neglect recombination--generation terms. We impose initial and mixed Dirichlet--Neumann boundary conditions:
\begin{align}
  n(\cdot,0) = n^0, \quad p(\cdot,0) = p^0, \quad D(\cdot,0)=D^0
  &\quad\mbox{in }\Omega, \label{1.ddic} \\
  n = \bar{n}, \quad p = \bar{p}, \quad V = \bar{V}
  &\quad\mbox{on }\Gamma_D,\ t>0, \label{1.dbc} \\
  J_n\cdot\nu = J_p\cdot\nu = \na V\cdot\nu = 0
  &\quad\mbox{on }\Gamma_N,\ t>0, \label{1.nbc} \\
  J_D\cdot\nu = 0 &\quad\mbox{on }\pa\Omega,\ t>0. \label{1.bcD}
\end{align}
The doping profile $A(x)$ and the initial and Dirichlet boundary data (except $\bar{V}$) are given, and $\lambda>0$ is the scaled Debye length. The boundary potential $\bar{V}$ is a function of the node potentials at the terminals, defining the applied bias. The Dirichlet conditions for the densities are obtained under the
assumptions of charge neutrality and thermal equilibrium \cite[Sec.~5.1]{Sel84} or from the first-order approximation of the Boltzmann semiconductor equation in the diffusion limit. Since the oxide vacancies are supposed not to leave the domain, we impose no-flux boundary conditions for $D$ on the whole boundary. These boundary conditions are usually used in the literature \cite{GSTD13,SBW09}. The mismatch of the boundary conditions for $n$, $p$ and for $D$ is one of the main mathematical difficulties.

\subsection{Network equations}

An electric network typically contains capacitors, inductors, and resistors. Capacitors store electric energy in an electric field and are used, for instance, to filter noise in a circuit, while inductors store electric energy in a magnetic field and oppose changes in the current flow. Resistors regulate the flow of the electric current and protect the circuit by limiting the current. The electric network considered here connects $n_C$ linear capacitors, $n_L$ inductors, and $n_R$ resistors as well as $n_I$ independent current and $n_V$ voltage sources. The network has $m$ nodes plus the ground node with zero potential. The variables are the node potentials $u(t)\in\R^m$, the currents $i_L(t)\in\R^{n_L}$ through inductors, the currents $i_V(t)\in\R^{n_V}$ through voltage sources, the independent current sources $i_I(t)\in\R^{n_I}$, and the independent voltage sources $v_V(t)\in\R^{n_V}$. The node potentials $u\in\R^m$ are related to the node potentials $u_D\in\R^2$ at the two terminals by the selection matrix $S=(S_{ij})\in\R^{m\times 2}$ via $u_D=S^Tu$. Here, the matrix $S$ relates the two device contacts to the $m$ network nodes and is defined by
\begin{align}\label{1.S}
  S_{ij} = \begin{cases}
  1 &\mbox{if the terminal }\Gamma_D^j\mbox{ touches the node }i, \\
  0 & \mbox{else}.
  \end{cases}
\end{align}

The topology of the network is described by the so-called incidence matrix that represents the graph of the electric circuit. Ordering the branches of the graph appropriately, we can decompose the incidence matrix into block form with the (reduced) incidence matrices $A_k\in\R^{m\times n_k}$, describing the branch-node relationships for capacitors, inductors, and resistors ($k=C,L,R$) and the relationships for voltage and current sources ($k=I,V$). The matrices $C\in\R^{n_C\times n_C}$, $L\in\R^{n_L\times n_L}$, and $R\in\R^{n_R\times n_R}$ denote the capacitance, inductance, and conductance matrices, respectively. In our simplified setting, these matrices are diagonal with the capacitance, inductance, and conductance values on the diagonals. In modified nodal analysis \cite{BaGu18}, the network equations can be formulated as
\begin{align}
  & A_CCA_C^T\frac{du}{dt} + A_RRA_R^Tu + A_Li_L + A_Vi_V 
  + I_M[J;u] = A_Ii_I,
  \label{1.dae1} \\
  & L\frac{di_L}{dt} = A_L^Tu, \quad -A_V^Tu = v_V, \label{1.dae2}
\end{align}
with the initial conditions
\begin{align}
  u(0) = u^0, \quad i_L(0) = i_L^0, \quad i_V(0) = i_V^0. \label{1.ic}
\end{align}
The current $I_M[J;u]\in\R^2$ represents the coupling with the memristor device and depends on the total current $J:=J_n+J_p+J_D$ and the node potentials $u$; see Section \ref{sec.coupl} for details. 

Equation \eqref{1.dae1} is a consequence of Kirchhoff's current law, according to which the sum of all inflowing branch currents is zero. The first term in \eqref{1.dae1} is the capacitor current, the second term is the resistor current according to Ohm's law, and the third and fourth terms are the currents through inductors and voltage sources, respectively. The first equation in \eqref{1.dae2} reflects the voltage-current relationship for inductors. The second equation in \eqref{1.dae2} follows from Kirchhoff's voltage law, according to which the sum of all branch voltages in a loop is zero. System \eqref{1.dae1}--\eqref{1.dae2} contains differential equations for $u$ and $i_L$ and additional algebraic equations, forming a system of differential-algebraic equations. The full model is a partial differential-algebraic system.

\subsection{Coupling conditions}\label{sec.coupl} 

We have two couplings, the device-to-circuit coupling and the circuit-to-device coupling. The total current through the terminal $\Gamma_D^j$ is given by
\begin{align*}
  I_D^j[J;u] = -\int_{\Gamma_D^j}(J-\lambda^2\pa_t\na V)\cdot\nu ds,
  \quad j=1,2,
\end{align*}
which is the sum of the total particle current and the so-called displacement current. It depends on the node potentials $u=(u_1,\ldots,u_m)$ via the Dirichlet boundary conditions for $V$. We show in Section \ref{sec.model} that $I_D^j$ can be formulated as a weighted sum of $du_j/dt$ and an integral over $\Omega$, depending on the total current $J$ and $u$; see \eqref{2.IDj}. The current $I_M[J;u]$ resulting from the device operation, and defined in \eqref{2.IM}, is an input in the circuit equations.

The node potentials on the terminals define the boundary potential
\begin{align*}
  \bar{V} = V_{\rm bi} + S^Tu\quad\mbox{on }\Gamma_D, 
\end{align*}
where the so-called built-in potential $V_{\rm bi}$ is the potential corresponding to the equilibrium densities and reads as $V_{\rm bi}=\operatorname{arcsinh}(A/(2n_i))$, where $n_i$ is the intrinsic density, being the product of the equilibrium electron and hole densities \cite{GSTD13}. As explained before, the node potential $u_D=S^Tu$ represents the applied bias. The full model is illustrated in Figure \ref{fig.coupled}.

\begin{figure}[ht]
\includegraphics[width=140mm]{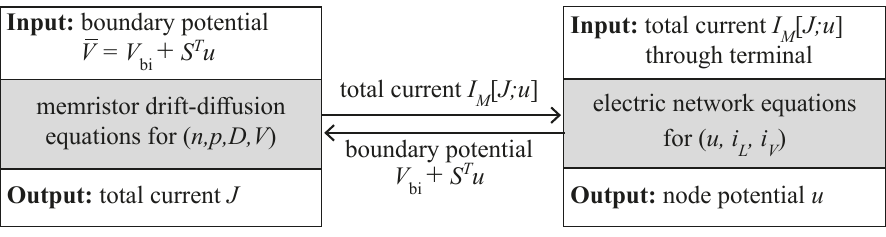}
\caption{Coupled drift--diffusion and electric network equations. The selection matrix $S$ is introduced in \eqref{1.S} and $I_M[J;u]$ is defined in \eqref{2.IM}.}
\label{fig.coupled}
\end{figure} 


\subsection{State of the art}

The semiconductor drift-diffusion equations for two species (electrons and holes) have been analyzed intensively since the 1980s; see, e.g., \cite{Gaj85,MRS90}. Fewer results are available for equations with more than two species. We refer, for instance, to \cite{BFS14,GlHu05} for an analysis in two space dimensions and to \cite{BFPR14} for a study in any space dimensions, assuming no-flux boundary conditions for the charge carriers and a Robin boundary condition for the electric potential. The memristor dynamics can be described by three-species drift--diffusion equations. The existence of weak solutions was proved in \cite{JJZ23}. The result was extended to the degenerate drift-diffusion equations in \cite{JuVe24} and to the Fermi--Dirac drift--diffusion model in \cite{HJP25}. 

Electric network models, coupled to two-species semiconductor equations, were analyzed in a series of papers \cite{ABG05,ABG10,ABGT03,ABR09,AlRo10}, but mostly in one space dimension or for steady states. The existence of solutions to the multidimensional drift-diffusion equations was shown in \cite{AlRo21}, but the proof contains a flaw; see Remark \ref{rem.mistake}. Thus, the existence for coupled circuit--device models in several space dimensions is an open problem, even for two species. We fill this gap in this paper for two-dimensional devices.

The complexity of the coupled partial-differential-algebraic equations (PDAEs) is measured by its index. Common indices are the differentiation index, which refers to the minimum number of differentiations required to obtain an explicit representation of the highest-order derivative, and the tractability index, which is based on a matrix chain construction. The tractability index for device--circuit PDAEs was analyzed in \cite{AlRo12,BoTi07}, and conditions were stated under which this index is one. For a review of PDAEs, we refer to \cite{BaGu18}. 

Electric circuits can also be coupled to the semiconductor energy-transport equations, which take into account the electron temperature \cite{AlCa07,BrJu08,BrJu11}. Since the analysis of the semiconductor equations is delicate and the electron temperature does not influence the lattice temperature significantly \cite{BrJu11}, we assume that the electron (and lattice) temperature is constant.

The novelty of this paper is the first proof of a global weak solution to the coupled circuit--device system for three species (including the case of two species).


\subsection{Assumptions}

We use the following notation. Let $M_k\in\R^{m\times n_k}$ for $k=1,\ldots,\ell$. We write $x\in \operatorname{ker}(M_1,\ldots,M_\ell)^T\subset\R^m$ if and only if $x\in\operatorname{ker}M_k^T$ for $k=1,\ldots,\ell$. Set $\Omega_T=\Omega\times(0,T)$ for $T>0$ and
\begin{align*}
  H_D^1(\Omega) = \big\{u\in H^1(\Omega):u=0\mbox{ on }\Gamma_D\big\},
  \quad H_D^{-1}(\Omega) = H_D^1(\Omega)'.
\end{align*}
We impose the following assumptions.
\begin{itemize}
\item[(A1)] Domain: $\Omega\subset\R^2$ is a bounded domain with Lipschitz boundary $\pa\Omega=\overline{\Gamma}_D\cup\overline{\Gamma}_N$ such that $\Gamma_N$ is relatively open and $\Gamma_D=\Gamma_D^1\cup\Gamma_D^2$. The terminals $\Gamma_D^1$ and $\Gamma_D^2$ are disjoint, relatively open, and have positive measure.
\item[(A2)] Data: $\lambda>0$, $A\in L^\infty(\Omega)$, and $i_I$, $v_V\in C^0([0,T])$. 
\item[(A3)] Initial data: $n^0$, $p^0$, $D^0\in L^2(\Omega)$ are nonnegative.
\item[(A4)] Boundary data: $\bar{n}$, $\bar{p}\in H^2(\Omega)\cap W^{1,\infty}(\Omega)$ are strictly positive and $V_{\rm bi}\in H^2(\Omega)\cap W^{1,\infty}(\Omega)$.
\item[(A5)] Index-1 topological conditions: $\operatorname{ker}(S,A_C,A_R,A_V)^T=\{0\}$ and $\operatorname{ker}Q_{CS}^TA_V=\{0\}$, where $Q_{CS}$ is the projection on $\operatorname{ker}(A_C,S)^T$.
\item[(A6)] Consistency: The initial data $(u^0,i_L^0,i_V^0)^T\in\R^{m+n_L+n_V}$ is consistent in the sense of \eqref{2.consist}.
\item[(A7)] Higher-order regularity: Let $f\in L^2(\Omega_T)$, $u^0$, $g\in L^2(\Omega)$. Then there exist unique solutions
$u\in L^2(0,T;H_D^1(\Omega)\cap H^2(\Omega))\cap H^1(0,T;L^2(\Omega))$ 
and $w\in H_D^1(\Omega)\cap H^2(\Omega)$ to the problems
\begin{align}
  & \pa_t u - \Delta u = f\ \mbox{in }\Omega,\ t>0, \quad
  u=0\ \mbox{on }\Gamma_D, \quad \na u\cdot\nu=0\ \mbox{on }
  \Gamma_N, \quad u(0)=u^0\ \mbox{in }\Omega, \label{1.mixed} \\
  & \Delta w = g\quad\mbox{in }\Omega, \quad 
  w=0\quad\mbox{on }\Gamma_D, \quad 
  \na w\cdot\nu=0\quad\mbox{on }\Gamma_N, \label{1.ellip}
\end{align}
satisfying the inequalities
\begin{align*}
  \|u\|_{L^2(0,T;H^2(\Omega))} + \|\pa_t u\|_{L^2(\Omega_T)}
  &\le C\|f\|_{L^2(\Omega_T)}+C\|u^0\|_{L^2(\Omega)}, \\
  \|w\|_{L^2(0,T;H^2(\Omega))} &\le C\|g\|_{L^2(\Omega_T)}.
\end{align*}
\end{itemize}

Assumptions (A1)--(A4) are conditions on the given data. The restriction to two space dimensions comes from the estimation of the quadratic drift terms, which is reduced to a linear expression by means of a logarithmic-type Gagliardo--Nirenberg inequality; see \eqref{3.2d} for details. The first condition in Assumption (A5) means that there are no LI-cutsets in the network, i.e.\ no cutsets consisting of inductances and/or current sources only. The second condition means that there are no CV-loops in the networks, i.e.\ no loops consisting of capacitances and voltage sources only. The consistency condition in Assumption (A6) comes from the differential-algebraic character of the network equations; see Section \ref{sec.decomp} for details.

Assumption (A7) is needed to obtain strong solutions to the approximated drift-diffusion and Poisson equations, which ensures the integrability of the normal derivative at the boundary, needed for the coupling with the electric network (see, for instance, \eqref{3.H2}). We note that the $H^2(\Omega)$ regularity for the electric potential was also needed in \cite[Sec.~3.2]{AlRo21} to estimate the drift term $n\na V$ in $L^2(\Omega)$. 

The regularity assumptions restrict the geometry of the domain. Indeed, it is well known that solutions to general linear elliptic problems with mixed boundary conditions satisfy $w\in W^{1,q_0}(\Omega)$ for some $q_0>2$ \cite{Gro94}. Under some conditions on the Dirichlet and Neumann boundary parts (in particular, $\Gamma_D$ and $\Gamma_N$ intersect with an ``angle'' not larger than $\pi$), the regularity improves to $q_0>3$ \cite[Theorem 4.8]{DiRe15}. If the domain is a two-dimensional polygon, precise regularity results can be found in \cite{Gri85}. Shamir's counterexample in \cite{Sha68} shows that $q_0\ge 4$ cannot be expected, even if the domain and the data are smooth. 

Assumption (A7) for the elliptic problem \eqref{1.ellip} is satisfied if the Dirichlet and Neumann boundaries do not meet or if they meet with inner angle $\omega<\pi/2$. By \cite[Sec.~II.3, Theorem~3.3]{Tem97}, the parabolic Neumann problem 
\begin{align*}
  \pa_t v - \Delta v = f\ \mbox{in }\Omega,\ t>0, \quad
  \na v\cdot\nu=0\ \mbox{on }\pa\Omega, \quad 
  v(0)=v^0\ \mbox{in }\Omega,
\end{align*}
possesses a unique strong solution if $f\in L^2(\Omega_T)$, $v^0\in H^1(\Omega)$ satisfying
\begin{align*}
  \|v\|_{L^2(0,T;H^2(\Omega))} 
  \le C\|f\|_{L^2(\Omega_T)}+C\|v^0\|_{L^2(\Omega)}.
\end{align*}
The regularity of parabolic mixed problems like \eqref{1.mixed} is more delicate. A result in this direction, but for a different parabolic problem, is due to \cite{HaRe09}, where it is proved that elliptic operators of divergence type fulfill maximal parabolic regularity. We consider the regularity for \eqref{1.mixed} as an assumption, since our focus is the understanding of the analytical treatment of the coupling between the semiconductor and network equations and not on the technical regularity issue.


\subsection{Main result and key ideas}

Set $\bar{D}=\exp(-\bar{V})$. We introduce the free energy 
\begin{align}\label{1.H}
  H[n,p,D,y] &= \int_\Omega\bigg\{n\bigg(\log\frac{n}{\bar{n}}-1\bigg)
  + p\bigg(\log\frac{p}{\bar{p}}-1\bigg)
  + D\bigg(\log\frac{D}{\bar{D}}-1\bigg)\bigg\}dx \\
  &\phantom{xx}+ \frac{\lambda^2}{2}\int_\Omega|\na(V-V_A)|^2dx
  + \frac12 y^T(E_1 - \pi^TSMS^T\pi)y, \nonumber 
\end{align}
where the first integral denotes the internal energy of the charge carriers, the second integral is the electric energy with the stationary potential \eqref{2.VA} (not depending on $y$), and the last term is the network energy. The matrices $E_1$ and $M$ are introduced in Lemma \ref{lem.decoupl} and \eqref{2.M}, respectively, $\pi$ is a projection matrix, defined in Section \ref{sec.decomp}, and $y=y(t)$ is the solution to the differential part of the network model; see Lemma \ref{lem.decoupl}. The matrix $E_1 - \pi^TSMS^T\pi$ is positive definite with respect to the variable $y$; see Lemma \ref{lem.E1}. Our main result is the existence of global solutions. 

\begin{theorem}[Global existence]\label{thm.ex}
Let Assumptions (A1)--(A7) hold and let $T>0$. Then there exists a unique weak solution to \eqref{1.dd1}--\eqref{1.bcD} on $[0,T]$ satisfying
\begin{align*}
  & n,\,p,\,D\in L^2(0,T;H^1(\Omega))\cap H^1(0,T;H^1(\Omega)'), 
  \quad y\in H^1(0,T;\R^{m+n_L}), \\
  & J_n,\ J_p,\,J_D\in L^2(\Omega_T), \quad
  V\in L^2(0,T;H^2(\Omega))\cap H^1(0,T;H^1(\Omega)').
\end{align*} 
Moreover, the following free energy inequality holds:
\begin{align*}
  H[&(n,p,D,y)(t)] + \frac12\int_0^t\int_\Omega
  |2\na\sqrt{n}-\sqrt{n}\na V|^2 dxd\tau \\
  &+ \frac12\int_0^t\int_\Omega
  |2\na\sqrt{p}+\sqrt{p}\na V|^2 dxd\tau 
  + \frac12\int_0^t\int_\Omega
  |2\na\sqrt{D}+\sqrt{D}\na V|^2 dxd\tau \le C,
\end{align*}
where $C>0$ depends on the time $T>0$ and the free energy at time $t=0$. 
\end{theorem}

To prove this theorem, we approximate the memristor equations \eqref{1.dd1}--\eqref{1.V} by truncating of the drift terms with parameter $k\in\N$ and apply the Leray--Schauder fixed-point theorem. The electric network equations are decomposed into a differential and an algebraic part; see Section \ref{sec.decomp}. The differential system is solved in the iteration argument simply by time integration. 

We need to find uniform estimates to pass to the limit $k\to\infty$ and to remove the truncation. These estimates are derived from the approximate free energy inequality. This proof requires some tools. First, we use test functions that are approximations of $\log n_k$ (and correspondingly for $p_k$ and $D_k$). The important point here is that the truncation, which is taken from \cite{JJZ23}, satisfies an approximate version of the chain rule $\sqrt{n_k}\na\log n_k=2\na\sqrt{n_k}$. The $H^2(\Omega)$ regularity from Assumption (A7) allows us to compute the dissipation of the coupling conditions, which in fact vanishes. Second, the drift terms are estimated independent of the truncation parameter $k$ by applying a logarithmic-type Gagliardo--Nirenberg inequality as in \cite[Sec.~3]{JJZ23}. At this point, we need the restriction to two space dimensions. 

The existence proof of \cite{AlRo21} for the coupled drift-diffusion electric-network model relies on a monotonicity property of the drift terms which is available for two species only. We refer to Remark \ref{rem.mistake} for the strategy of the proof in \cite{AlRo21}.

Our second main result are positive lower and upper bounds for the densities.

\begin{theorem}[Lower and upper bounds]\label{thm.bd}
Let $n^0$, $p^0$, $D^0\in L^\infty(\Omega)$ satisfy $n^0\ge c_0$, $p^0\ge c_0$, $D^0\ge c_0$ in $\Omega$ and let $\bar{n}\ge c_0$, $\bar{p}\ge c_0$, $\bar{D}\ge c_0$ on $\Gamma_D\times(0,T)$ for some $c_0>0$. Then there exist constants $c>0$ and $C>0$ depending on $T>0$ such that
\begin{align*}
  c\le n(t)\le C, \quad c\le p(t)\le C, \quad
  c\le D(t)\le C \quad\mbox{for }0\le t\le T.
\end{align*}
\end{theorem}

The $L^\infty(\Omega)$ bounds follow from \cite[Theorem 1.2]{JJZ23}, while the lower bounds are new. They are proved by a Stampacchia truncation argument. The upper bound depends on the end time $T$ because of the application of Gronwall's lemma. The lower bound depends on the upper bound and in particular heavily on $T$. The determination of upper and lower bounds that are uniform in time seems to require some assumptions on the Dirichlet boundary data. This is subject to future work.

The paper is organized as follows. Section \ref{sec.model} is concerned with details on the modeling of the coupled system, making precise the coupling conditions, the decomposition of the differential-algebraic system in a differential and an algebraic part, and the full model of semiconductor equations, decoupled network equations, and coupling conditions. Theorem \ref{thm.ex} is proved in Section \ref{sec.proof}, while the proof of Theorem \ref{thm.bd} is presented in Section \ref{sec.bd}. The paper ends with some conclusions in Section \ref{sec.con}.


\section{Details on the modeling}\label{sec.model}

We detail the coupling conditions between the device and circuit models,  describe the decomposition of the circuit model into a differential and an algebraic part, and summarize the full model.

\subsection{Current through the terminals}

The total current of the device through the terminal $\Gamma_D^j$ is given by the sum of the particle currents and the displacement current, which takes into account the change of the electric field with respect to time,
\begin{align*}
  I_D^j = -\int_{\Gamma_D^j}(J_n+J_p+J_D
  - \lambda^2\pa_t\na V)\cdot\nu d\sigma,
  \quad j=1,2.
\end{align*}
As in \cite[Sec.~2.3]{AlRo21}, it is convenient to write the total current density $J:=J_n+J_p+J_D$ as an integral over $\Omega$. For this, we introduce the auxiliary functions $w_i\in H^1(\Omega)$ as the unique solution \cite[Theorem 4.10]{McL00} to
\begin{align*}
  \Delta w_i = 0 \quad\mbox{in }\Omega, \quad  
  w_i = \delta_{ij}\quad\mbox{on }\Gamma_D^j,\ j=1,2, \quad
  \na w_i\cdot\nu = 0 \quad\mbox{on }\Gamma_N.
\end{align*}
Adding equations \eqref{1.dd1} and taking into account the Poisson equation \eqref{1.V}, we obtain
\begin{align*}
  \diver(J-\lambda^2\pa_t\na V)
  = \pa_t(n-p-D) - \lambda^2\pa_t\Delta V = 0 \quad\mbox{in }\Omega.
\end{align*}
Hence, by the divergence theorem (at least formally),
\begin{align*}
  I_D^1+I_D^2 = \int_\Omega\diver(J-\lambda^2\pa_t\na V)dx = 0.
\end{align*}
This corresponds to the physical intuition that the ingoing current equals the outgoing current. Moreover, by the divergence theorem again,
\begin{align*}
  I_D^j = -\int_{\pa\Omega}w_j(J-\lambda^2\pa_t\na V)\cdot\nu ds
  = -\int_\Omega \na w_j\cdot(J-\lambda^2\pa_t\na V)dx.
\end{align*}
The advantage of this formulation is that only the regularity $J\in L^2(\Omega_T)$ and $\pa_t V\in L^2(0,T;H^1(\Omega))$ is needed, while higher regularity is required to define the integral of $J\cdot\nu$ over $\Gamma_D$. 

Following \cite[Sec.~2.3]{AlRo21}, we decompose the electric potential to make explicit the dependence of the current on the external applied potential. Observing that $A$ and $V_{\rm bi}$ are time-independent, we introduce the stationary part $V_A$ of the electric potential as the unique solution to
\begin{align}\label{2.VA}
  \lambda^2\Delta V_A = A\quad\mbox{in }\Omega, \quad
  V_A = V_{\rm bi}\quad\mbox{on }\Gamma_D, \quad
  \na V_A\cdot\nu=0 \quad\mbox{on }\Gamma_N.
\end{align}
Furthermore, we introduce the linear solution operator $\mathcal{L}:H^{-1}_D(\Omega)\to H^1(\Omega)$, $\mathcal{L}[g]=f$, where $f\in H_D^1(\Omega)$ is the unique solution to
\begin{align}\label{2.L}
  \lambda^2\Delta f = g\quad\mbox{in }\Omega, \quad
  f=0\quad\mbox{on }\Gamma_D, \quad 
  \na f\cdot\nu=0\quad\mbox{on }\Gamma_N.
\end{align}
Then the electric potential can be written as
\begin{align}\label{2.V}
  V = V_A + \sum_{k=1}^2 w_ku_{D}^k + \mathcal{L}[n-p-D]
  \quad\mbox{in }\Omega,
\end{align}
where $u_D=(u_D^1,u_D^2)^T$. Differentiating this expression and taking into account that $V_A$ is time-independent, the displacement current density becomes
\begin{align*}
  \pa_t\na V = \sum_{k=1}^2\na w_k\frac{du_D^k}{dt} 
  + \na\mathcal{L}[\pa_t(n-p-D)] 
  = \sum_{k=1}^2\na w_k\frac{du_D^k}{dt} 
  + \na\mathcal{L}[\diver J].
\end{align*}
We insert this formulation into the total current $I_D^j$:
\begin{align}\label{2.IDj}
  & I_D^j = \sum_{k=1}^2 M_{jk}\frac{du_D^k}{dt}
  + \mathcal{I}_j[J;u_D], \quad\mbox{where} \\
  & \mathcal{I}_j[J;u_D] = -\int_\Omega\na w_j\cdot\big(J-\lambda^2\na\mathcal{L}[\diver J]\big)dx,
  \nonumber \\
  & M_{jk} = \lambda^2\int_\Omega\na w_j\cdot\na w_k dx\quad
  \mbox{for }j,k=1,2. \nonumber 
\end{align}
It follows for weak solutions satisfying $J\in L^2(\Omega_T)$ that $\diver J\in L^2(0,T;H^{-1}_D(\Omega))$ such that $\mathcal{I}_j[J;u_D]$ is well defined. The functional $\mathcal{I}[J;u_D]$ depends on $u_D$ through the potential $V$; see \eqref{2.V}. 

We claim that the matrix $M=(M_{jk})\in\R^{2\times 2}$ is positive definite. Indeed, the sum $w_1+w_2$ solves the elliptic problem 
\begin{align*}
  \Delta(w_1+w_2)=0\quad\mbox{in }\Omega, \quad
  w_1+w_2=1\quad\mbox{on }\Gamma_D, \quad
  \na(w_1+w_2)\cdot\nu=0\quad\mbox{on }\Gamma_N.
\end{align*}
We infer that $w_1+w_2=1$ in $\Omega$ from which we conclude that $\mathcal{I}_1[J;u_D]+\mathcal{I}_2[J;u_D]=0$ and $M_{j1}+M_{j2}=0$ for $j=1,2$. We obtain $M_{11}=-M_{12}=-M_{21}=M_{22}$ and hence
\begin{align}\label{2.M}
  M = \lambda^2\|\na w_1\|_{L^2(\Omega)}^2\begin{pmatrix}
  1 & -1 \\ -1 & 1 \end{pmatrix}.
\end{align}
Consequently, $M$ is positive definite.

To formulate \eqref{2.IDj} in a more compact form, we set $\mathcal{I}=(\mathcal{I}_1,\mathcal{I}_2)^T\in\R^2$ and $I_D=(I_D^1,I_D^2)^T\in\R^2$. Recall definition \eqref{1.S} of the selection matrix $S$. Then $u_D=S^Tu\in\R^2$ and 
\begin{align}\label{2.IM}
  I_M[J;u] := S(I_D^1,I_D^2)^T = SMS^T\frac{du}{dt} + S\mathcal{I}[J;S^Tu]\in\R^m.
\end{align}

We conclude that the full model consists of the Poisson--drift--diffusion equations \eqref{1.dd1}--\eqref{1.bcD}, the network equations \eqref{1.dae1}--\eqref{1.ic}, and the coupling conditions 
\begin{align*}
  \bar{V} = V_{\rm bi} + S^Tu, \quad 
  I_M[J;u] = SMS^T\frac{du}{dt} + S\mathcal{I}[J;S^Tu].
\end{align*}


\subsection{Decomposition into a differential and an algebraic part}
\label{sec.decomp}

Introducing the variable $x=(u,i_L,i_V)$, system \eqref{1.dae1}--\eqref{1.dae2} can be written more compactly as
\begin{align}\label{2.dxdt}
  E\frac{dx}{dt} = Ax + F[J;x] + s(t)\quad\mbox{for }t>0, \quad  
  x(0) = x^0,
\end{align}
where $x^0=(u^0,i_L^0,i_V^0)\in\R^n$,
\begin{align*}
  E = \begin{pmatrix}
  A_CCA_C^T+SMS^T & 0 & 0 \\
  0 & L & 0 \\
  0 & 0 & 0
  \end{pmatrix}, \quad
  A = \begin{pmatrix}
  -A_RRA_R^T & -A_L & -A_V \\
  A_L^T & 0 & 0 \\
  A_V^T & 0 & 0 
  \end{pmatrix}\in\R^{n\times n},
\end{align*}
and
\begin{align}\label{2.sF}
  s(t) = \begin{pmatrix}
  A_I i_I(t) \\ 0 \\ v_V(t)
  \end{pmatrix}, \quad
  F[J;x] = -\pi^TS\mathcal{I}[J;S^T\pi x]\in\R^n.
\end{align}
We have set $n=m+n_L+n_V$, and $\pi x=\pi(u,i_L,i_V)^T=u$ is the projection matrix on the first block. The network variable $x(t)$ should not be confused with the spatial variable $x\in\Omega$. We keep this double notation, since it is standard in electric network theory and the meaning can easily be understood from the context.

We decompose \eqref{2.dxdt} in such a way that the differential and algebraic components decouple. To this end, we decompose $\R^n=\operatorname{ker}E\oplus\operatorname{ran}E$. Let $P:\R^n\to\operatorname{ker}E$ be the projector onto $\operatorname{ker}E$ and $Q=I-P$ be the projector onto $\operatorname{ran}E$. For given $x\in\R^n$, we set $y:=Px$ and $z:=Qx$, which means that $x=y+z$. The projection $Q$ can be written more explicitly as
\begin{align*}
  Q = \begin{pmatrix}
  Q_{CS} & 0 & 0 \\ 0 & 0 & 0 \\ 0 & 0 & I
  \end{pmatrix},
\end{align*}
where $Q_{CS}$ is the projection on $\operatorname{ker}(A_C,S)^T$. Observe that $F$ does not depend on $z$, since
\begin{align}\label{2.STpiQ}
  S^T\pi Q = \begin{pmatrix} S^TQ_{CS} & 0 & 0 \end{pmatrix} = 0
\end{align}
implies that $S^T\pi x = S^T\pi Px + S^t\pi Qx = S^T\pi y$. This means that $F$ depends on $y$ only. Therefore, we write, slightly abusing the notation, $F[J;y]:=F[J;x]$, where $y=Px$. 

\begin{lemma}\label{lem.decoupl}
Let the index-1 conditions in Assumption (A5) hold. Then \eqref{2.dxdt} is equivalent to
\begin{align}\label{2.yz}
  \frac{dy}{dt} = PE_1^{-1}(A_1y+F[J;y]+s(t)), \quad
  z = QE_1^{-1}(A_1y+s(t)), \quad y(0)=y^0,
\end{align}
where $E_1:=E-AQ$, $A_1:=AP$, and $y^0:=Px^0$. In particular, it holds that $QE_1^{-1}F[J;y]=0$. 
\end{lemma} 

The lemma is proved in \cite[Sec.~2.4]{AlRo21}. The invertibility of $E_1$ follows directly from the index-1 conditions. The differential equation for $y$ in Lemma \ref{lem.decoupl} is decoupled from the algebraic equation for $z$ and can be solved independently from the latter equation. The initial values $y(0)=Px^0$ and $z(0)=Qx^0$ have to satisfy a compatibility condition, namely the consistency condition 
\begin{align}\label{2.consist}
  Qx^0 = z(0) = QE_1^{-1}(A_1y^0+s(0)) = QE_1^{-1}(A_1Px^0+s(0)),
\end{align} 
which is a condition on the initial data \eqref{1.ddic}, imposed in Assumption (A6).  

We define the boundary potential in terms of $y$. Since $Q_{CS}$ projects onto $\operatorname{ker}(A_C,S)^T$ by definition, we have $S^TQ_{CS}=0$ and hence $S^T\pi y=S^T\pi Px=S^T\pi(x-Qx)=S^T(u-Q_{CS}u)=S^Tu$. Consequently,
\begin{align*}
  V = V_{\rm bi} + u_D = V_{\rm bi} + S^Tu = V_{\rm bi} + S^T\pi y
  \quad\mbox{on }\Gamma_D.
\end{align*}

Finally, we prove that $P^TE_1P$ is symmetric and positive definite on $\operatorname{ker}Q$.

\begin{lemma}\label{lem.E1}
The matrix $P^TE_1P$ is symmetric and there exists $c>0$ such that $(Px)^T E_1Px\ge c|Px|^2$ for all $x\in\R^{n}$.
\end{lemma}

\begin{proof}
It follows from $QP=0$ and the definition of $E$ that
\begin{align*}
  (P^TE_1P)^T = P^T(E-AQ)^TP = P^TEP - P^TQ^TA^TP = P^TEP = P^TE_1P,
\end{align*}
showing the symmetry of $P^TE_1P$. The matrix $P^TEP=(I-Q)^TE(I-Q)$ can be computed explicitly:
\begin{align*}
  P^TEP = \begin{pmatrix}
  (I-Q_{CS})^T(A_CCA_C^T+SMS^T)(I-Q_{CS}) & 0 & 0 \\
  0 & L & 0 \\
  0 & 0 & 0 \end{pmatrix}.
\end{align*}
Let $x=(x_u,x_L,x_V)^T\in\R^n$. The matrices $C$ and $L$ are diagonal with positive diagonal entries, so they are positive definite. Then, since $M$ is positive definite, we conclude that $A_CCA_C^T+SMS^T$ is positive definite. We infer that there exists $c>0$ such that
\begin{align*}
  (Px)^TE_1Px &= x^TP^TEPx 
  = x_u^T(I-Q_{CS})^T(A_CCA_C^T+SMS^T)(I-Q_{CS})x_u
  + x_L^TLx_L \\
  &\ge c(|(I-Q_{CS})x_u|^2 + |x_L|^2) = c|Px|^2,
\end{align*}
where the last step follows from $Px=((I-Q_{CS})x_u,-x_L,0)^T$.
\end{proof}

Lemma \ref{lem.E1} implies that the matrix $E_1 - \pi^TSMS^T\pi$ appearing in the free energy \eqref{1.H} is positive definite in the sense 
\begin{align*}
  y^T(E_1 - \pi^TSMS^T\pi)y \ge y^TE_1y = x^TP^TE_1Px
  \ge c|Px|^2 = c|y|^2 \quad\mbox{for }y=Px\in\R^n.
\end{align*}


\subsection{Precise model equations}

The full model consists of three parts. The first part are the memristor drift-diffusion equations \eqref{1.dd1}--\eqref{1.V} for the densities $n$, $p$, $D$ and the electric potential $V$, together with the initial and mixed boundary conditions \eqref{1.ddic}--\eqref{1.bcD}. In particular, the potential satisfies $V=\bar{V}= V_{\rm bi} + S^T\pi y$ on $\Gamma_D$, which couples the circuit to the device. The second part are the network equations \eqref{1.dae1}--\eqref{1.dae2} with the initial conditions \eqref{1.ic}, containing the coupling $I_M[J;u]$. In the proof, we work only with the decomposed system \eqref{2.yz} consisting of the projection $y=Px$, where $x=(u,i_L,i_V)^T$, and the algebraic part $z$, which depends on $y$ and $s(t)=(A_Ii_I(t),0,v_V(t))$. The third part are the coupling conditions. The circuit-to-device coupling is realized by 
\begin{align*}
  \bar{V}(t) = V_{\rm bi} + (S^Tu(t))_j 
  \quad\mbox{on }\Gamma_D^j,\ t>0,
\end{align*}
and the built-in potential $V_{\rm bi}$ is time-independent.
The device-to-circuit coupling is defined by 
\begin{align*}
  I_M[J;u] = SMS^T\frac{dy}{dt} + S\mathcal{I}[J;S^Tu],
\end{align*}
where $\mathcal{I}[J;S^Tu]$ is introduced in \eqref{2.IDj}.


\section{Proof of the main theorem}\label{sec.proof}

The proof is based on an approximate problem, the Leray--Schauder fixed-point theorem, and a compactness argument. More precisely, we solve first the approximate problem including a truncation operator. The free energy inequality yields a priori estimates uniform in the approximation parameter. These bounds are sufficient to apply the Aubin--Lions compactness lemma to pass to the limit of vanishing approximation parameter.

\subsection{Solution of an approximate drift-diffusion system}

We define the approximate problem by truncating the drift terms via $T_k(s)=\max\{0,\min\{k,s\}\}$ for $s\in\R$ and $k\in\N$. The approximate problem is then defined by
\begin{align}
  \pa_t n &= \diver(\na n - T_k(n)\na V), \label{3.nk} \\
  \pa_t p &= \diver(\na p + T_k(p)\na V), \label{3.pk} \\
  \pa_t D &= \diver(\na D + T_k(D)\na V), \label{3.Dk} \\
  \lambda^2\Delta V &= n-p-D+A(x) \quad\mbox{in }\Omega,\ t>0,
  \label{3.Vk} \\
  \frac{dy}{dt} &= PE_1^{-1}(A_1y + F[J;y] + s(t))
  \quad\mbox{for }t>0, \label{3.yk}
\end{align}
together with the initial and boundary conditions \eqref{1.ddic}--\eqref{1.bcD} and \eqref{1.ic} (i.e.\ $y(0)=Px^0$). The device-to-circuit coupling is realized by the current density $J=J_n+J_p+J_D$ in $F[J;y]$, and the circuit-to-device coupling appears in the time-dependent Dirichlet boundary condition for the electric potential via $\bar{V}(t)=V_{\rm bi}+S^T\pi y(t)$ on $\Gamma_D$. 

\begin{lemma}[Existence for the approximate system]\label{lem.approx}
There exist $q_0>2$ and a global strong solution $(n_k,p_k,D_k,V_k,y_k)$ to the approximate problem \eqref{3.nk}--\eqref{3.yk} with the initial and boundary conditions \eqref{1.ddic}--\eqref{1.bcD} and $y(0)=Px^0$ satisfing $n_k\ge 0$, $p_k\ge 0$, $D_k\ge 0$ in $\Omega_T$, and
\begin{align*}
  & n_k,\,p_k,\,D_k\in H^1(0,T;L^2(\Omega))\cap L^2(0,T;H^2(\Omega)), 
  \quad y_k\in H^1([0,T];\R^{m+n_L}), \\
  & V_k\in H^1(0,T;H^1(\Omega))\cap L^2(0,T;H^2(\Omega))\cap
  L^\infty(0,T;W^{1,q_0}(\Omega)).
\end{align*}
\end{lemma}

\begin{proof}
The proof is similar to the proof of \cite[Lemma 2.1]{JJZ23}, but we need to include the differential equation for $y_k$ and the coupling conditions. Therefore, we present a full proof.

{\em Step 1: Definition of the fixed-point operator.}
Let $\sigma\in[0,1]$ and $(n^*,p^*,D^*,y^*)\in L^2(\Omega_T;\R^3)\times L^2(0,T;\R^{m+n_L})$ be given. We define the linearized approximate problem 
\begin{align}
  \pa_t n &= \diver(\na n - \sigma T_k(n^*)\na V), \label{3.nstar} \\
  \pa_t p &= \diver(\na p + \sigma T_k(p^*)\na V), \label{3.pstar} \\
  \pa_t D &= \diver(\na D + \sigma T_k(D^*)\na V), \label{3.Dstar} \\
  \lambda^2\Delta V &= n^*-p^*-D^*+A(x) \quad\mbox{in }\Omega,\ t>0,
  \label{3.Vstar} \\
  \frac{dy}{dt} &= \sigma PE_1^{-1}(A_1y^* + F[J;y^*] + s(t))
  \quad\mbox{for }t>0, \label{3.ystar}
\end{align}
with the initial and boundary conditions
\begin{align}
  n(0)=\sigma n^0,\ p(0)=\sigma p^0,\ D(0)=\sigma D^0
  &\quad\mbox{in }\Omega, \label{3.icstar} \\
  n=\sigma\bar{n},\quad p=\sigma\bar{p}, \quad V=\bar{V}^* 
  &\quad\mbox{on }\Gamma_D,\ t>0, \label{3.bc1star} \\
  \na n\cdot\nu = \na p\cdot\nu = \na V\cdot\nu = 0
  &\quad\mbox{on }\Gamma_N,\ t>0, \label{3.bc2star} \\
  \na D\cdot\nu = 0 &\quad\mbox{on }\pa\Omega,\ t>0, \label{3.bcD} \\
  y(0)= \sigma Px^0, & \nonumber 
\end{align}
where $\bar{V}^* = V_{\rm bi}+S^T\pi Py^*$. Here, we have
\begin{align*}
  J &= (\na n - \sigma T_k(n^*)\na V) + (\na p + \sigma T_k(p^*)\na V)
  + (\na D + \sigma T_k(D^*)\na V), \\
  F[J;y] &= -\pi^T S\mathcal{I}[J;S^T\pi Py]
  = -\pi^TS\bigg(\int_\Omega\na w_j\cdot\big(J-\lambda^2
  \na\mathcal{L}[\diver J]\big)dx\bigg)_{j=1,2}.
\end{align*}

There exists a unique weak solution $V$ to \eqref{3.Vstar} with the boundary conditions for $V$ in \eqref{3.bc1star} and \eqref{3.bc2star}. Then, by \cite[Theorem 23.A]{Zei90}, the linear parabolic equations \eqref{3.nstar}--\eqref{3.Dstar} with the associated initial and boundary conditions possesses a unique weak solution $(n,p,D)$. Thanks to the truncation, we have $J\in L^2(\Omega_T)$ such that $F[J;y]$ is well-defined. The solution $y\in H^1(0,T)$ is obtained by simply integrating the differential equation \eqref{3.ystar} with $y$-independent right-hand side over $(0,t)$. This defines the fixed-point operator
\begin{align*}
  & S:L^2(\Omega_T;\R^3)\times L^2(0,T;\R^{m+n_L})\times[0,1]\to 
  L^2(\Omega_T;\R^3)\times L^2(0,T;\R^{m+n_L}), \\
  & (n^*,p^*,D^*,y^*;\sigma)\mapsto (n,p,D,y).
\end{align*}
It holds that $S(n^*,p^*,D^*,y^*;0)=0$, and standard arguments show that $S$ is continuous. As the range of $S$ is contained in 
\begin{align*}
  \big(L^2(0,T;H^1(\Omega;\R^3))\cap H^1(0,T;H^{-1}_D(\Omega;\R^3))\big)
  \times H^1(0,T),
\end{align*}
the Aubin--Lions lemma implies that $S$ is a compact operator. It remains to derive uniform estimates for any fixed point of $S(\cdot,\cdot,\cdot;\sigma)$. 

{\em Step 2: Uniform estimates.} Let $(n,p,D,y)\in L^2(\Omega_T;\R^3)\times L^2(0,T;\R^{m+n_L})$ be a fixed point of $S(\cdot,\cdot,\cdot;\sigma)$. Using the test function $V-\bar{V}$ in the weak formulation of \eqref{3.Vstar} with $(n^*,p^*,D^*)=(n,p,D)$ yields, after standard estimations, that
\begin{align}\label{3.estV}
  \|\na V\|_{L^2(\Omega)}^2 \le C\big(1 + \|n\|_{L^2(\Omega)}^2
  + \|p\|_{L^2(\Omega)}^2 + \|D\|_{L^2(\Omega)}^2 + |y(t)|^2\big).
\end{align}
The dependence on $y(t)$ comes from the boundary condition $\bar{V} = V_{\rm bi}+S^T\pi Py(t)$. The test function $n-\sigma \bar{n}$ in the weak formulation of \eqref{3.nstar} with $n^*=n$ gives, again after standard computations,
\begin{align}\label{3.estn}
  \|n&(t)\|_{L^2(\Omega)}^2 + \int_0^t\|\na n\|_{L^2(\Omega)}^2 d\tau \\
  &\le C_1 + C_1\int_0^t\big(\|n\|_{L^2(\Omega)}^2 
  + \|p\|_{L^2(\Omega)}^2 + \|D\|_{L^2(\Omega)}^2 
  + |y(\tau)|^2\big)d\tau. \nonumber 
\end{align}
The constant $C_1>0$ depends on $k$ but not on $\sigma$. Similar estimates, with the same constant $C_1$, hold for $p$ and $D$. 

We wish to estimate $|y(t)|^2$. For this, 
using $Py=P^2x=Px=y$ and the symmetry of $P^TE_1P$ (see Lemma \ref{lem.E1}), we compute
\begin{align*}
  \frac12\frac{d}{dt}(y^TE_1y) &= \frac12\frac{d}{dt}\big(
  y^T(P^TE_1P)y\big) = y^T(P^TE_1P)\frac{dy}{dt}
  = y^TE_1\frac{dy}{dt} \\
  &= \sigma y^TE_1PE_1^{-1}\big(A_1y + F[J;y] + s(t)\big).
\end{align*}
We integrate this equation over $(0,t)$ and use the positive definiteness of $P^TE_1P$ on a subspace (see Lemma \ref{lem.E1}):
\begin{align}\label{3.y2}
  |y(t)|^2 &= |Px(t)|^2 \le Cx(t)^T P^TE_1Px(t) = Cy(t)^TE_1y(t) \\
  &\le Cy(0)^TE_1y(0) + C\int_0^t|y(\tau)|^2d\tau 
  + C\int_0^t |y(\tau)||F[J(\tau),y(\tau)]|d\tau \nonumber \\
  &\phantom{xx}+ C\int_0^t |y(\tau)||s(\tau)|d\tau \nonumber \\
  &\le C\sigma^2 x^0P^TE_1Px^0 + C\int_0^t|y(\tau)|^2 d\tau
  + C\int_0^t|\pi^TS\mathcal{I}[J(\tau);S^T\pi y(\tau)]|^2d\tau 
  \nonumber \\
  &\phantom{xx}+ C\int_0^t |y(\tau)||s(\tau)|d\tau \nonumber \\
  &\le C(x^0,s) + C\int_0^t|y(\tau)|^2 d\tau
  + C\int_0^t|\mathcal{I}[J(\tau);S^T\pi y(\tau)]|^2 d\tau. \nonumber 
\end{align}

To estimate $I[J,S^T\pi y]$, we need the $H^2(\Omega)$ regularity for linear parabolic problems, supposed in Assumption (A7), and the $W^{1,q_0}(\Omega)$ regularity for linear elliptic problems from \cite{Gro89}. This yields $V\in L^\infty(0,T;W^{1,q_0}(\Omega))$ and $n\in L^2(0,T;H^2(\Omega))\hookrightarrow L^2(0,T;W^{1,q}(\Omega))$ for any $q<\infty$ (since $d\le 2$), and therefore $\na T_k(n)\cdot\na V\in L^2(\Omega_T)$. With $J_n^k = \na n-T_k(n)\na V$, this shows that
\begin{align*}
  \diver J_n^k = \Delta n - \na T_k(n)\cdot\na V - T_k(n)\Delta V
  \in L^2(\Omega_T).
\end{align*}
In particular, we can interpret the no-flux boundary condition $J_n\cdot\nu=0$ on $\Gamma_N$ in the sense of the dual of the Lions--Magenes space $H_{00}^{1/2}(\Gamma_N)'$ \cite[Sec.~18]{BaCa84}. We use $f(t):=\mathcal{L}[\diver J_n^k(t)]\in H_D^1(\Omega)$ for a.e.\ $t\in(0,T)$ as a test function in the weak formulation of \eqref{2.L} and integrate by parts \cite[Theorem 18.9]{BaCa84} (here we need $\diver J_n\in L^2(\Omega_T)$):
\begin{align*}
  \lambda^2\int_\Omega|\na f(t)|^2 dx
  &= -\int_\Omega f(t)\diver J_n^k dx
  = \int_\Omega\na f(t)\cdot J_n^k dx \\
  &\le \frac{\lambda^2}{2}\int_\Omega|\na f(t)|^2 dx
  + \frac{1}{2\lambda^2}\int_\Omega |J_n^k(t)|^2 dx.
\end{align*}
This gives
\begin{align*}
  \|\na\mathcal{L}[\diver J_n^k(t)]\|_{L^2(\Omega)}
  \le \lambda^{-2}\|J_n^k(t)\|_{L^2(\Omega)}.
\end{align*}
Analogous inequalities hold for $J_p^k=-(\na p+T_k(p)\na V)$ and $J_D^k=-(\na D+T_k(D)\na V)$. By Definition \eqref{2.IDj} of $\mathcal{I}$ and recalling that $J=J_n^k+J_p^k+J_D^k$, 
\begin{align}\label{2.Ij}
  \int_0^t\big|\mathcal{I}_j[J;S^T\pi y]\big|^2 d\tau
  &\le \int_0^t\bigg|\int_\Omega 
  \na w_j\cdot(J-\lambda^2\na\mathcal{L}[\diver J])dx\bigg|^2 d\tau \\
  &\le C\|\na w_j\|_{L^2(\Omega)}^2\int_0^t\|J\|_{L^2(\Omega)}^2 d\tau
  \le C\int_0^t\|J\|_{L^2(\Omega)}^2 d\tau, \nonumber 
\end{align}
By estimate \eqref{3.estV} for $\na V$, we have
\begin{align*}
  \int_0^t\big|\mathcal{I}_j[J;S^T\pi y]\big|^2 d\tau
  &\le C\int_0^t\big(\|\na n\|_{L^2(\Omega)}^2 
  + \|\na p\|_{L^2(\Omega)}^2 + \|\na D\|_{L^2(\Omega)}^2 
  + \|\na V\|_{L^2(\Omega)}^2\big)d\tau \\
  &\le C\int_0^t\big(\|\na n\|_{L^2(\Omega)}^2 
  + \|\na p\|_{L^2(\Omega)}^2 + \|\na D\|_{L^2(\Omega)}^2\big)d\tau \\
  &\phantom{xx}+ C\int_0^t\big(\|n\|_{L^2(\Omega)}^2
  + \|p\|_{L^2(\Omega)}^2 + \|D\|_{L^2(\Omega)}^2 + |y(t)|^2\big)d\tau,
\end{align*} 
We conclude from \eqref{3.y2} that
\begin{align}\label{3.esty}
  |y(t)|^2 &\le C(x^0,s) 
  + C_2\int_0^t\big(\|\na n\|_{L^2(\Omega)}^2 
  + \|\na p\|_{L^2(\Omega)}^2 + \|\na D\|_{L^2(\Omega)}^2\big)d\tau \\
  &\phantom{xx}+ C_2\int_0^t\big(\|n\|_{L^2(\Omega)}^2
  + \|p\|_{L^2(\Omega)}^2 + \|D\|_{L^2(\Omega)}^2 
  + |y(\tau)|^2\big)d\tau. \nonumber 
\end{align}
We multiply \eqref{3.estn} for $n$, and the corresponding estimates for $p$ and $D$, by $2C_2$ and add them to \eqref{3.esty}. Then the gradient terms in \eqref{3.esty} can be absorbed, and we end up with
\begin{align*}
  2&C_2\big(\|n(t)\|_{L^2(\Omega)}^2 + \|p(t)\|_{L^2(\Omega)}^2 
  + \|D(t)\|_{L^2(\Omega)}^2\big) + |y(t)|^2 \\
  &\phantom{xx}+ C_2\int_0^t\big(\|\na n\|_{L^2(\Omega)}^2 
  + \|\na p\|_{L^2(\Omega)}^2 + \|\na D\|_{L^2(\Omega)}^2\big)d\tau \\
  &\le 2C_1C_2 + C(x^0,s) 
  + (1+2C_1)C_2\int_0^t\big(\|n\|_{L^2(\Omega)}^2 
  + \|p\|_{L^2(\Omega)}^2 + \|D\|_{L^2(\Omega)}^2 
  + |y(\tau)|^2\big)d\tau.
\end{align*}
Gronwall's inequality shows that $(n,p,D)$ is uniformly bounded in $L^\infty(0,T;L^2(\Omega))$ (and thus also in $L^2(\Omega_T)$) and $y$ is uniformly bounded in $L^\infty(0,T;\R^{m+n_L})$. This gives the desired bounds (uniform in $\sigma$), and the Leray--Schauder fixed-point theorem shows the existence of a weak solution to \eqref{3.nk}--\eqref{3.yk} with the corresponding initial and boundary conditions. 

It remains to verify the nonnegativity of $n$, $p$, and $D$. This is done by using $n^-:=\min\{0,n\}$ as an admissible test function in the weak formulation of \eqref{3.nk}: 
\begin{align*}
  \frac12\frac{d}{dt}\int_\Omega |n(t)^-|^2dx
  + \int_\Omega|\na n^-|^2 dx 
  = \int_\Omega T_k(n)\mathrm{1}_{\{n<0\}}\na V\cdot\na n dx = 0,
\end{align*}
from which we infer that $n(t)^-=0$ and hence $n(t)\ge 0$ in $\Omega$. Similar arguments show that $p(t)\ge 0$ and $D(t)\ge 0$ in $\Omega$ for $t>0$. This finishes the proof.
\end{proof}


\subsection{Approximate free energy inequality}

The estimates derived in the proof of Lemma \ref{lem.approx} for $(n_k,p_k,D_k,V_k,y_k)$ depend on the truncation parameter $k$ (the constant $C_1$ depends on $k$). We derive bounds uniform in $k$ from the free energy $H$, defined in \eqref{1.H}. As the densities are only nonnegative, according to Lemma \ref{lem.approx}, we cannot use $\log n_k$ as a test function, and we need to regularize the free energy functional. To this end, we introduce for $k\in\N$ and $\alpha>0$ the auxiliary functions
\begin{align*}
  g_\alpha^k(u) &= \int_0^u\int_1^v\frac{dwdv}{T_k(w)+\alpha} 
  \quad\mbox{for }u\ge 0, \\
  G_\alpha^k(u|\bar{u}) &= g_\alpha^k(u) - g_\alpha^k(\bar{u})
  - (g_\alpha^k)'(\bar{u})(u-\bar{u})\quad\mbox{for }u,\bar{u}\ge 0, \\
  h_\alpha^k(u) &= \int_0^u\frac{dw}{\sqrt{T_k(w)+\alpha}}
  \quad\mbox{for }u\ge 0.
\end{align*}
The function $g_\alpha^k(u)$ approximates $u(\log u-1)$, the Bregman distance $G_\alpha^k(u|\bar{u})$ approximates $u\log(u/\bar{u})-u+\bar{u}$, and $h_\alpha^k(u)$ approximates $2\sqrt{u}$. The auxiliary functions are defined in such a way that the following chain rules hold:
\begin{equation}\label{3.chain}
\begin{aligned}
  \na (g_\alpha^k)'(n_k) &= (g_\alpha^k)''(n_k)\na n_k
  = \frac{\na n_k}{T_k(n_k)+\alpha}, \\
  \na h_\alpha^k(n_k) &= (h_\alpha^k)'(n_k)\na n_k
  = \frac{\na n_k}{\sqrt{T_k(n_k)+\alpha}}.
\end{aligned}
\end{equation}
The regularized free energy is then defined by
\begin{align*}
  & H_\alpha^k[n,p,D,V,y] = H_{\alpha,1}^k + H_{\alpha,2}^k
  + H_{\alpha,3}^k, \quad\mbox{where} \\
  & H_{\alpha,1}^k = \int_\Omega\big(G_\alpha^k(n|\bar{n})
  + G_\alpha^k(p|\bar{p}) + G_\alpha^k(D|\bar{D})\big)dx, \\
  & H_{\alpha,2}^k = \frac{\lambda^2}{2}\int_\Omega
  |\na(V-V_A)|^2dx, \\
  & H_{\alpha,3}^k = \frac12 y^T(E_1-\pi^TSMS^T\pi)y.
\end{align*}

\begin{lemma}[Regularized free energy inequality I]
Let $(n_k,p_k,D_k,V_k,y_k)$ be a solution to the approximate problem \eqref{3.nk}--\eqref{3.yk} with the corresponding initial and boundary conditions. Then there exists a constant $C>0$ independent of $\alpha$ and $k$ such that
\begin{align*}
  H_\alpha^k&[(n_k,p_k,D_k,V_k,y_k)(t)]
  + \frac12\int_0^t\int_\Omega\big|\na h_\alpha^k(n_k)
  - \sqrt{T_k(n_k)+\alpha}\na V_k\big|^2 dxd\tau \\
  &+ \frac12\int_0^t\int_\Omega\big|\na h_\alpha^k(p_k)
  + \sqrt{T_k(p_k)+\alpha}\na V_k\big|^2 dxd\tau \\
  &+ \frac12\int_0^t\int_\Omega\big|\na h_\alpha^k(D_k)
  + \sqrt{T_k(D_k)+\alpha}\na V_k\big|^2 dxd\tau
  \le C + CH_\alpha^k[n^0,p^0,D^0,Px^0].
\end{align*}
\end{lemma}

\begin{proof}
We introduce the current densities
\begin{align*}
  J_n^k = \na n_k - T_k(n_k)\na V_k, \quad
  J_p^k = -(\na p_k+T_k(p_k)\na V_k), \quad
  J_D^k = -(\na D_k+T_k(D_k)\na V_k).
\end{align*}
The first term $H_{\alpha,1}^k$ can be computed as in the proof of \cite[Lemma 6]{JJZ23}, giving
\begin{align}\label{3.H1}
  \frac{dH_{\alpha,1}^k}{dt} &= -\int_\Omega\Big(
  J_n^k\cdot\big(\na(g_\alpha^k)'(n_k) - \na(g_\alpha^k)'(\bar{n})\big)
  - J_p^k\cdot\big(\na(g_\alpha^k)'(p_k) 
  - \na(g_\alpha^k)'(\bar{p})\big) \\
  &\phantom{xx}- J_D^k\cdot\big(\na(g_\alpha^k)'(D_k) 
  - \na(g_\alpha^k)'(\bar{D})\big)\Big)dx. \nonumber 
\end{align}
The calculation of $H_{\alpha,2}^k$ is different from that one in the proof of \cite[Lemma 6]{JJZ23}, since we need to include the time-dependent Dirichlet boundary data. Observe that, in view of Assumption (A7), $\pa_t\na V_k(t)\in H^1(\Omega)$ and therefore $\pa_t\na V_k(t)\cdot\nu|_{\Gamma_D}\in L^2(\Gamma_D)$, and recall that $V_k=V_A+u_{D,k}$ on $\Gamma_D$, $\na V_k\cdot\nu=0$ on $\Gamma_N$, where $V_A$ is the stationary potential defined in \eqref{2.VA} and $u_{D,k}=S^T\pi Py_k$. Then, using the Poisson equation, we obtain
\begin{align*}
  \frac{dH_{\alpha,2}^k}{dt} &= \lambda^2\int_\Omega
  \na(V_k-V_A)\cdot\pa_t\na V_k dx \\
  &= -\int_\Omega(V_k-V_A)\pa_t(n_k-p_k-D_k)dx
  + \lambda^2\int_{\pa\Omega}(V_k-V_A)\pa_t\na V_k\cdot\nu d\sigma \\
  &= -\int_\Omega(V_k-V_A)\diver(J_n^k+J_p^k+J_D^k)dx
  + \lambda^2\int_{\Gamma_D}u_D^k\pa_t\na V_k\cdot\nu d\sigma.
\end{align*}
We integrate by parts in the first term on the right-hand side, giving
\begin{align}
  \frac{dH_{\alpha,2}^k}{dt} 
  &= \int_\Omega\na(V_k-V_A)\cdot(J_n^k+J_p^k+J_D^k)dx
  - \int_{\Gamma_D}u_D^k(J_n^k+J_p^k+J_D^k-\lambda^2\pa_t\na V_k)
  \cdot\nu d\sigma \nonumber \\
  &= \int_\Omega\na(V_k-V_A)\cdot(J_n^k+J_p^k+J_D^k)dx
  + u_D^k\cdot I_D^k, \label{3.H2}
\end{align}
since $J_D^k\cdot\nu=0$ on $\pa\Omega$ and where the total current of the device $I_D^k=(I_D^{k,1},I_D^{k,2})$ is given by
\begin{align*}
  I_D^{k,j} = -\int_{\Gamma_D^j}(J_n^k+J_p^k+D_D^k
  -\lambda^2\pa_t\na V_k)\cdot\nu d\sigma.
\end{align*}
We add expressions \eqref{3.H1} and \eqref{3.H2}:
\begin{align}\label{3.H1H2}
  \frac{d}{dt}(H_{\alpha,1}^k+H_{\alpha,2}^k)
  &= -\int_\Omega\Big(J_n^k\cdot\na\big((g_\alpha^k)'(n_k)-V_k\big)
  + J_p^k\cdot\na\big((g_\alpha^k)'(p_k)+V_k\big) \\
  &\phantom{xx}+ J_D^k\cdot\na\big((g_\alpha^k)'(D_k)+V_k\big)\Big)dx
  + \int_\Omega\Big(J_n^k\cdot\na\big((g_\alpha^k)'(\bar{n})-V_A\big) 
  \nonumber \\
  &\phantom{xx}+ J_p^k\cdot\na\big((g_\alpha^k)'(\bar{p})+V_A\big)
  + J_D^k\cdot\na\big((g_\alpha^k)'(\bar{D})+V_A\big)\Big)dx
  + u_D^k\cdot I_D^k. \nonumber 
\end{align}

Taking into account the chain rule \eqref{3.chain}, the gradient $\na(g_\alpha^k)'(n_k)$ can be reformulated, and as in the proof of \cite[Lemma 2.2]{JJZ23}, we end up after an application of Young's inequality with
\begin{align}\label{3.H12}
  \frac{d}{dt}(H_{\alpha,1}^k+H_{\alpha,2}^k)
  &= -\frac12\int_\Omega\big|\na h_\alpha^k(n_k)
  - \sqrt{T_k(n_k)+\alpha}\na V_k\big|^2 dx \\
  &\phantom{xx}- \frac12\int_\Omega\big|\na h_\alpha^k(p_k)
  + \sqrt{T_k(p_k)+\alpha}\na V_k\big|^2 dx \nonumber \\
  &\phantom{xx}- \frac12\int_\Omega\big|\na h_\alpha^k(D_k)
  - \sqrt{T_k(D_k)+\alpha}\na V_k\big|^2 dx \nonumber \\
  &\phantom{xx}+ C\int_\Omega|\na(V_k-V_A)|^2 dx 
  + C\int_\Omega|\na V_A|^2 dx \nonumber \\
  &\phantom{xx}+ C\int_\Omega\big(T_k(n_k)+T_k(p_k)+T_k(D_k)\big)dx
  + u_D^k\cdot I_D^k, \nonumber 
\end{align}
where the constant $C>0$ depends on the $L^\infty(\Omega)$ norms of $\na((g_\alpha^k)'(\bar{n})-V_A)$, $\na((g_\alpha^k)'(\bar{p})+V_A)$, and $\na((g_\alpha^k)'(\bar{D})+V_A)$.

We claim that the last term on the right-hand side of \eqref{3.H12} is canceled from $dH_{\alpha,3}^k/dt$. Indeed, by the symmetry of $E_1$ and $M$ and equation \eqref{3.yk},
\begin{align}\label{3.H3}
  \frac{dH_{\alpha,3}^k}{dt} 
  &= y_k^T(E_1-pi^TSMS^T\pi)\frac{dy_k}{dt} \\
  &= y_k^TE_1PE_1^{-1}\big(A_1y_k + F[J;y_k] + s(t)\big)
  - y_k^T\pi^TSMS^T\pi\frac{dy_k}{dt}. \nonumber 
\end{align}
We rewrite the last term. Because of $S^{T}\pi Q=0$ (see \eqref{2.STpiQ}) and with $y_k=Px_k$,
\begin{align*}
  S^T\pi\frac{dx_k}{dt} = S^T\pi(P+Q)\frac{dx_k}{dt}
  = S^T\pi P\frac{dx_k}{dt} = S^T\pi\frac{dy_k}{dt}.
\end{align*} 
We infer that
\begin{align*}
  u_D^k\cdot I_D^k &= (S^T\pi x_k)^TI_D^k = x_k^T\pi^TSI_D^k
  = x_k^T\pi^TS\bigg(M\frac{du_D^k}{dt} 
  + \mathcal{I}[J_k;u_D^k]\bigg) \\
  &= x_k^T\pi^TSM\frac{d}{dt}(S^T\pi x_k) 
  +  x_k^T\pi^TS \mathcal{I}[J_k;u_D^k] \\
  &= x_k^T\pi^TSMS^T\pi\frac{dy_k}{dt} + y_k^TF[J_k;y_k].
\end{align*}
We replace the last term in \eqref{3.H3} by the previous expression, leading to
\begin{align*}
  \frac{dH_{\alpha,3}^k}{dt} 
  = y_k^TE_1PE_1^{-1}\big(A_1y_k + F[J_k;y_k] + s(t)\big)
  - u_D^k\cdot I_D^k + y_k^TF[J_k;y_k].
\end{align*}
It follows from Lemma \ref{lem.decoupl} that $E_1PE_1^{1}F[J_k;y_k] = E_1(P+Q)E_1^{-1}F[J_k;y_k] = F[J_k;y_k]$, from which we infer that
\begin{align*}
  \frac{dH_{\alpha,3}^k}{dt} 
  = y_k^TE_1PE_1^{-1}(A_1y_k+s(t)) - u_D^k\cdot I_D^k.
\end{align*}

Adding the previous expression to inequality \eqref{3.H1H2}, the term $u_D^k\cdot I_D^k$ cancels, giving
\begin{align*}
  \frac{dH_{\alpha}^k}{dt} 
  &= \frac{d}{dt}(H_{\alpha,1}^k+H_{\alpha,2}^k + H_{\alpha,3}^k)
  = -\frac12\int_\Omega\big|\na h_\alpha^k(n_k)
  - \sqrt{T_k(n_k)+\alpha}\na V_k\big|^2 dx \\
  &\phantom{xx}- \frac12\int_\Omega\big|\na h_\alpha^k(p_k)
  + \sqrt{T_k(p_k)+\alpha}\na V_k\big|^2 dx \nonumber \\
  &\phantom{xx}- \frac12\int_\Omega\big|\na h_\alpha^k(D_k)
  + \sqrt{T_k(D_k)+\alpha}\na V_k\big|^2 dx \nonumber \\
  &\phantom{xx}+ C\int_\Omega|\na(V_k-V_A)|^2 dx 
  + C\int_\Omega|\na V_A|^2 dx \nonumber \\
  &\phantom{xx}+ C\int_\Omega\big(T_k(n_k)+T_k(p_k)+T_k(D_k)\big)dx
  + y_k^TE_1PE_1^{-1}(A_1y_k+s(t)).
\end{align*}
Now, observing that \cite[(2.9)]{JJZ23}
\begin{align*}
  \int_\Omega(T_k(n_k)+T_k(p_k)+T_k(D_k))dx 
  \le CH_\alpha^k[n_k,p_k,D_k,V_k] + C
\end{align*}
and that $ y_k^TE_1PE_1^{-1}(A_1y_k+s(t))\le C|y_k|^2 + C$, we find that 
\begin{align*}
  \frac{dH_{\alpha}^k}{dt} 
  &= -\frac12\int_\Omega\big|\na h_\alpha^k(n_k)
  - \sqrt{T_k(n_k)+\alpha}\na V_k\big|^2 dx \\
  &\phantom{xx}- \frac12\int_\Omega\big|\na h_\alpha^k(p_k)
  + \sqrt{T_k(p_k)+\alpha}\na V_k\big|^2 dx \\
  &\phantom{xx}- \frac12\int_\Omega\big|\na h_\alpha^k(D_k)
  + \sqrt{T_k(D_k)+\alpha}\na V_k\big|^2 dx 
  + CH_\alpha^k + C|y_k|^2 + C.
\end{align*}
Since $E_1-\pi^TSMS^T\pi$ is positive definite with respect to $y_k$, we have $c|y_k|^2\le H_\alpha^k[n_k,p_k,D_k,$ $V_k]$, and an application of Gronwall's lemma finishes the proof.
\end{proof}

The limit $\alpha\to 0$ leads to a free energy inequality with truncation. For this, we introduce the functions
\begin{align*}
  g_k(u) &= \int_0^u\int_1^v\frac{dwdv}{T_k(v)} 
  \quad\mbox{for }u\ge 0, \\
  G_k(u|\bar{u}) &= g^k(u) - g^k(\bar{u}) - (g^k)'(\bar{u})(u-\bar{u})
  \quad\mbox{for }u,\bar{u}\ge 0, \\
  h_k(u) &= \int_0^u\frac{dw}{\sqrt{T_k(w)}}\quad\mbox{for }u\ge 0,
\end{align*}
and the free energy with truncation
\begin{align*}
  H_k[n,p,D,V,y] &= \int_\Omega\big(G_k(n|\bar{n}) 
  + G_k(p|\bar{p}) + G_k(D|\bar{D})\big)dx \\
  &\phantom{xx}+ \frac{\lambda^2}{2}\int_\Omega|\na(V-V_A)|^2 dx
  + \frac12y^T(E_1-\pi^TS^TMS\pi)y.
\end{align*}
The following lemma can be proved as in \cite[Lemma 2.3]{JJZ23}.

\begin{lemma}[Regularized free energy inequality II]\label{lem.regei2}
Let $(n_k,p_k,D_k,V_k,y_k)$ be a solution to the approximate problem \eqref{3.nk}--\eqref{3.yk} with the corresponding initial and boundary conditions. Then there exists a constant $C>0$ independent of $k$ such that
\begin{align}\label{3.eik}
  H_k&[(n,p,D,V,y)(t)]
  + \frac12\int_0^t\int_\Omega\big|\na h_k(n_k)
  - \sqrt{T_k(n_k)}\na V_k\big|^2 dxd\tau \\
  &+ \frac12\int_0^t\int_\Omega\big|\na h_k(p_k)
  + \sqrt{T_k(p_k)}\na V_k\big|^2 dxd\tau \nonumber \\
  &+ \frac12\int_0^t\int_\Omega\big|\na h_k(D_k)
  + \sqrt{T_k(D_k)}\na V_k\big|^2 dxd\tau
  \le C + CH_k[n^0,p^0,D^0,Px^0]. \nonumber 
\end{align}
\end{lemma}


\subsection{Uniform bounds}

The regularized free energy inequality \eqref{3.eik} implies the following $k$-independent bounds that can be proved as in \cite[Lemma 2.4]{JJZ23}.

\begin{lemma}[Uniform estimates]
The following bounds hold with a constant $C>0$ that is independent of $k$:
\begin{align*}
  \|g_k(n_k)\|_{L^\infty(0,T;L^1(\Omega))}
  + \|g_k(p_k)\|_{L^\infty(0,T;L^1(\Omega))}
  + \|g_k(D_k)\|_{L^\infty(0,T;L^1(\Omega))} &\le C, \\
  \|n_k\log n_k\|_{L^\infty(0,T;L^1(\Omega))}
  + \|p_k\log p_k\|_{L^\infty(0,T;L^1(\Omega))}
  + \|D_k\log D_k\|_{L^\infty(0,T;L^1(\Omega))} &\le C, \\
  \big\|\sqrt{T_k(n_k)}\na V_k\big\|_{L^\infty(0,T;L^1(\Omega))}
  + \big\|\sqrt{T_k(p_k)}\na V_k\big\|_{L^\infty(0,T;L^1(\Omega))} \\
  \phantom{xx}
  + \big\|\sqrt{T_k(D_k)}\na V_k\big\|_{L^\infty(0,T;L^1(\Omega))}
  &\le C, \\
  \|h_k(n_k)\|_{L^2(0,T;W^{1,1}(\Omega))}
  + \|h_k(p_k)\|_{L^2(0,T;W^{1,1}(\Omega))}
  + \|h_k(D_k)\|_{L^2(0,T;W^{1,1}(\Omega))} &\le C.
\end{align*} 
\end{lemma}

In two space dimensions, we can estimate the drift terms and derive uniform estimates for the densities in the $L^\infty(0,T;L^2(\Omega))$ norm. The key difficulty is the reduction of the cubic drift term with the test function $n_k$ to a quadratic term. This is done by applying elliptic $W^{1,q_0}(\Omega)$ regularity and a logarithmic-type Gagliardo--Nirenberg inequality \cite[(22)]{BHN94}, leading to
\begin{align}\label{3.2d}
  \bigg|\int_\Omega T_k(n_k)\na V_k\cdot\na n_k dx\bigg|
  \le \delta \|\na n_k\|_{L^2(\Omega)}^{1+(3d-2)/(d+2)} + C(\delta),
\end{align}
where $\delta>0$ is arbitrary. The exponent on the right-hand side is not larger than two if and only if $d\le 2$. In this case, the gradient term can be absorbed by the diffusion term. This explains the restriction to two space dimensions. With this idea, we can prove the following result. 

\begin{lemma}[Uniform $L^2(\Omega)$ bounds]\label{lem.L2}
Let $d\le 2$. There exists a constant $C>0$ independent of $k$ such that
\begin{align*}
  \|n_k\|_{L^\infty(0,T;L^2(\Omega))}^2
  &+ \|p_k\|_{L^\infty(0,T;L^2(\Omega))}^2
  + \|D_k\|_{L^\infty(0,T;L^2(\Omega))}^2 \\
  &+ \int_0^T\big(\|\na n_k\|_{L^2(\Omega)}^2 
  + \|\na p_k\|_{L^2(\Omega)}^2 + \|\na D_k\|_{L^2(\Omega)}^2\big)dx
  \le C.
\end{align*}
\end{lemma}

\begin{proof}
The lemma can be proved as \cite[Lemma 3.1]{JJZ23}. The only difference is the dependence of the boundary data $\bar{V}$ on $y(t)$. However, the free energy inequality of Lemma \ref{lem.regei2} provides a $k$-uniform $L^\infty(0,T)$ bound for $y_k$ such that $\bar{V}$ is bounded too. Thus, we can proceed as in the proof of \cite[Lemma 3.1]{JJZ23}.
\end{proof}

\begin{lemma}[Uniform bound for the displacement current]
There exists a constant $C>0$ independent of $k$ such that
\begin{align*}
  \|J_n^k\|_{L^2(\Omega_T)} + \|J_p^k\|_{L^2(\Omega_T)}
  + \|J_D^k\|_{L^2(\Omega_T)} &\le C, \\
  \bigg\|\frac{dy_k}{dt}\bigg\|_{L^2(0,T)}
  + \|F[J_k;y_k]\|_{L^2(0,T)} &\le C.
\end{align*}
\end{lemma}

\begin{proof}
By elliptic regularity \cite{Gro89}, we have
\begin{align*}
  \|V_k\|_{W^{1,q_0}(\Omega)} \le C\big(\|n_k\|_{L^r(\Omega)}
  + \|p_k\|_{L^r(\Omega)} + \|D_k\|_{L^r(\Omega)} + |y_k(t)|\big)
  \le C,
\end{align*}
where $r=2q_0/(2+q_0)\in(1,2)$, and we have used the uniform bounds in Lemma \ref{lem.L2}. Then, by the H\"older inequality with $1/\rho+1/q_0=1/2$ (or $\rho=2q_0/(q_0-2)>2$) and the classical Gagliardo--Nirenberg inequality,
\begin{align*}
  \int_0^T\|J_n^k\|_{L^2(\Omega)}^2 dt
  &\le \int_\Omega\big(\|\na n_k\|_{L^2(\Omega)}^2
  + \|n_k\|_{L^\rho(\Omega)}^2\|\na V_k\|_{L^{q_0}(\Omega)}^2\big)dt \\
  &\le C\int_0^T\big(\|\na n_k\|_{L^2(\Omega)}^2 
  + \|n_k\|_{H^1(\Omega)}^{2-2/\rho}\|n_k\|_{L^1(\Omega)}^{2+2/\rho}
  \|\na V_k\|_{L^{q_0}(\Omega)}^2\big)dt \\
  &\le C + C\int_0^T\|n_k\|_{H^1(\Omega)}^{2-2/\rho}dt \le C,
\end{align*}
where the last step follows from the uniform bounds for $n_k$ in $L^2(0,T;H^1(\Omega))$ and $L^\infty(0,T;$ $L^1(\Omega))$ as well as the $L^2(0,T;L^{p_0}(\Omega))$ bound for $\na V_k$. Similar estimates hold for $J_p^k$ and $J_D^k$. 

Next, taking into account definition \eqref{2.sF} of $F[J_k;y_k]$ and inequality \eqref{2.Ij},
\begin{align*}
  \|F[J_k;y_k]\|_{L^2(0,T)}^2
  &\le C\int_0^T|\mathcal{I}[J_k;S^T\pi y_k]|^2 dt \\
  &\le C\int_0^T\big(\|J_n^k\|_{L^2(\Omega_T)}^2 
  + \|J_p^k\|_{L^2(\Omega_T)}^2 + \|J_D^k\|_{L^2(\Omega_T)}^2\big)dt
  \le C.
\end{align*}
It follows from $\|s\|_{L^\infty(0,T)}\le C$ (see Assumption (A2)), $\|y_k\|_{L^\infty(0,T)}\le C$ (by the free energy inequality), and Lemma \ref{lem.decoupl} that
\begin{align*}
  \bigg\|\frac{dy_k}{dt}\bigg\|_{L^2(0,T)}^2
  &= \int_0^T\big|PE_1^{-1}(A_1y_k+F[J_k;y_k]+s(t))\big|^2 dt \\
  &\le C\big(\|y_k\|_{L^2(0,T)}^2 + \|F[J_k;y_k]\|_{L^2(0,T)}^2
  + \|s\|_{L^2(0,T)}^2\big)dt \le C.
\end{align*}
This finishes the proof.
\end{proof}


\subsection{Limit $k\to\infty$}

The limit $k\to\infty$ is simpler than in \cite{JJZ23} since we have global uniform bounds. In fact, since $(J_n^k)$ is bounded in $L^2(\Omega_T)$, $(\pa_t n_k)$ is bounded in $L^2(0,T;H^{-1}_D(\Omega))$. This, together with the uniform $L^2(0,T;H^1(\Omega))$ bound for $(n_k)$, allows us to apply the Aubin--Lions lemma to conclude that, up to a subsequence, as $k\to\infty$,
\begin{align*}
  n_k\to n,\quad p_k\to p, \quad D_k\to D
  \quad\mbox{strongly in }L^2(\Omega_T). 
\end{align*}
This implies that $T_k(n_k) - n = (T_k(n_k)-n_k)+(n_k-n)\to 0$ strongly in $L^1(\Omega_T)$, since
\begin{align*}
  \int_0^T\int_\Omega&|T_k(n_k)-n_k|dxdt
  = \int_0^T\int_{\{n_k>k\}}(n_k-k)dxdt \\
  &\le \int_0^T\int_{\{n_k>k\}}n_k\frac{\log n_k}{\log k}dxdt
  \le \frac{1}{\log k}\|n_k\log n_k\|_{L^\infty(0,T;L^1(\Omega))}\to 0.
\end{align*}
By weak compactness, we also have
\begin{align*}
  \na n_k\rightharpoonup \na n &\quad\mbox{weakly in }L^2(\Omega_T), \\
  \pa_t n_k\rightharpoonup \pa_t n &\quad\mbox{weakly in }
  L^2(0,T;H^{-1}_D(\Omega)).
\end{align*}

Our goal is to show that $J_n^k\rightharpoonup J_n=\na n-n\na V$ weakly in $L^2(\Omega_T)$. For this, we observe that $h_k(u)=2\sqrt{u}$ for $u\le k$ and $h_k(u)=u/\sqrt{k}+\sqrt{k}$ else, which which we deduce that
\begin{align*}
  \sup_{0<t<T}&\int_\Omega|h_k(n_k)-2\sqrt{n_k}|^2 dx
  = \sup_{0<t<T}\int_{\{n_k>k\}}\bigg|\frac{n_k}{\sqrt{k}}
  + \sqrt{k} - 2\sqrt{n_k}\bigg|^2 dxdt \\
  &= \frac{1}{k}\sup_{0<t<T}\int_{\{n_k>k\}}(\sqrt{n_k}-\sqrt{k})^4
  dxdt \le \frac{C}{k}\to 0,
\end{align*}
where the constant $C>0$ depends on the $L^\infty(0,T;L^2(\Omega))$ norm of $n_k$. It follows from the triangle inequality that $h_k(n_k)\to 2\sqrt{n}$ strongly in $L^2(\Omega_T)$. We have even convergence of the gradients:
\begin{align*}
  \|\na(h_k(n_k)-2\sqrt{n})\|_{L^2(0,T;H^1(\Omega)')}
  \le \|h_k(n_k)-2\sqrt{n}\|_{L^2(\Omega_T)}\to 0.
\end{align*}

Concerning the electric potential, we know that $\na V_k\rightharpoonup^*\na V$ weakly* in $L^\infty(0,T;L^2(\Omega))$. Then we infer from $\sqrt{T_k(n_k)}\to \sqrt{n}$ strongly in $L^2(\Omega_T)$ that $\sqrt{T_k(n_k)}\na V_k\rightharpoonup \sqrt{n}\na V$ weakly in $L^2(0,T;L^1(\Omega))$. Consequently,
\begin{align*}
  \na h_k(n_k)-\sqrt{T_k(n_k)}\na V_k
  \rightharpoonup 2\na\sqrt{n} - \sqrt{n}\na V
  \quad\mbox{weakly in }L^2(0,T;H^1(\Omega)').
\end{align*}
On the other hand, the free energy inequality shows that $\xi_k:=\na h_k(n_k)-\sqrt{T_k(n_k)}\na V_k$ is uniformly bounded in $L^2(\Omega_T)$ such that, up to a subsequence, $\xi_k\rightharpoonup \xi$ weakly in $L^2(\Omega_T)$. This shows that $\xi=2\na\sqrt{n} - \sqrt{n}\na V$. 

The convergence $\sqrt{T_k(n_k)}\to \sqrt{n}$ strongly in $L^2(\Omega_T)$  and the chain rule $\sqrt{T_k(n_k)}\na h_k(n_k)$ $=\na n_k$ imply that
\begin{align*}
  J_n^k &= \sqrt{T_k(n_k)}\big(\na h_k(n_k)-\sqrt{T_k(n_k)}\na V_k\big)
  + \na n_k - \sqrt{T_k(n_k)}\na h_k(n_k) \\
  &\to \sqrt{n}\big(2\na\sqrt{n}-\sqrt{n}\na V\big) 
  = \na n-n\na V =: J_n \quad\mbox{weakly in }L^1(\Omega_T),
\end{align*}
and because of the uniform $L^2(\Omega_T)$ bound for $J_n^k$, this weak convergence  also holds in $L^2(\Omega_T)$. Thus, passing to the limit $k\to\infty$ in the weak formulation
\begin{align*}
  \int_0^T\langle \pa_t n_k,\phi\rangle dt
  = \int_0^T\int_\Omega(\na n_k-T_k(n_k)\na V_k)\cdot\na\phi dxdt
\end{align*}
for $\phi\in L^2(0,T;H^1_D(\Omega))$, we find that
\begin{align*}
  \int_0^T\langle\pa_t n,\phi\rangle dt
  = \int_0^T\int_\Omega(\na n-n\na V)\cdot\na\phi dxdt.
\end{align*}
The convergences for the other charge carriers are proved in an analogous way.

The $H^1(0,T)$ bound for $(y_k)$ and the compactness of the embedding $H^1(\Omega)\hookrightarrow L^2(0,T)$ imply the existence of a subsequence (not relabeled) such that
\begin{align*}
  y_k\to y\quad\mbox{strongly in }L^2(0,T), \quad
  \frac{dy_k}{dt}\rightharpoonup \frac{dy}{dt}\quad
  \mbox{weakly in }L^2(0,T).
\end{align*}
Recall, by definitions \eqref{2.IDj} and \eqref{2.sF}, that
\begin{align*}
  F[J_k;y_k]_j = -\pi^TS\mathcal{I}_j[J_k;S^T\pi y_k]
  = -\pi^TS\int_\Omega\na w_j\cdot
  (J_k-\lambda^2\mathcal{L}[\diver J_k])dx.
\end{align*}
It follows from $J_k=J_n^k+J_p^k+J_D^k\rightharpoonup J:=J_n+J_p+J_D$ weakly in $L^2(\Omega_T)$ and the boundedness of the linear operator $\mathcal{L}$ that $\na\mathcal{L}[\diver J_k]\rightharpoonup\na\mathcal{L}[\diver J]$ weakly in $L^2(\Omega_T)$. We conclude that
\begin{align*}
  F[J_k;y_k] \rightharpoonup F[J,y]\quad\mbox{weakly in }L^2(0,T).
\end{align*}
Thus, passing to the limit $k\to \infty$ in
\begin{align*}
  \frac{dy_k}{dt} = PE_1^{-1}(A_1y_k+s(t)) + F[J_k;y_k],
\end{align*}
we see that $y$ solves $dy/dt=PE_1^{-1}(A_1y+s(t))+F[J;y]$. 

Since $(n_k)$ is bounded in $L^2(\Omega_T)$, the elliptic regularity implies a uniform bound for $V_k$ in $L^2(0,T;H^2(\Omega))$, showing that the limit $V$ is an element of $L^2(0,T;H^2(\Omega))$. 

Finally, the free energy inequality follows after passing to the limit $k\to\infty$ in \eqref{3.eik} using the weak convergences and the lower weakly semicontinuity of the norms. This ends the proof.

\begin{remark}[Existence proof in \cite{AlRo21}]\label{rem.mistake}\rm
A coupled two-species semiconductor device--electric network model was studied in \cite{AlRo21}. The local ex\-istence proof is based on the Banach fixed-point theorem. With the fixed-point operator $S$, one needs to estimate $S(n^1,p^1,y^1) - S(n^2,p^2,y^2)$, where $(n^j,p^j)\in (\bar{n},\bar{p})+L^2(0,T;H^1_D(\Omega;\R^2))$. In particular, we have to estimate $\mathcal{L}[\diver(J_n^1-J_n^2)]$, where $J^j=\na n^j-n^j\na V^j$. It is argued in \cite[p.~6129]{AlRo21} that $\mathcal{L}[\Delta(n^1-n^2)] = -\lambda^{-2}(n^1-n^2)$, since $\na(n^1-n^2)\cdot\nu=0$ on $\Gamma_N$ (in a generalized sense). Unfortunately, this boundary condition is generally not satisfied. Our proof avoids this argument, since we derive an estimate for fixed points, in particular for $\mathcal{L}[\diver J_n]$, and $J_n$ satisfies the corresponding boundary conditions by construction. 
\qed\end{remark}


\section{Proof of Theorem \ref{thm.bd}}\label{sec.bd} 

The upper bound is proved in \cite[Theorem 1.2]{JJZ23} by using an Alikakos-type iteration argument. The idea is to derive $L^q(\Omega)$ bounds uniform in $q\in\N$ and to pass to the limit $q\to\infty$. The main difficulty is the estimation of the drift terms. This is done by applying, similar as in the existence proof for the approximate problem, the logarithmic-type Gagliardo--Nirenberg inequality \cite[(22)]{BHN94}, which restricts the proof to two space dimensions. It remains to show the positive lower bounds. To this end, we use the test function $(n-m)^-=\min\{0,n-m\}$ in the weak formulation of \eqref{1.dd1}, where $m=m_0\exp(-\mu t)$ and $m_0=\min\{\inf_\Omega n^0,\inf_{\Gamma_D}\bar{n}\}>0$:
\begin{align*}
  \frac12\frac{d}{dt}&\int_\Omega(n-m)^-(t)^2 dx
  + \int_\Omega|\na(n-m)^-|^2 dx - \int_\Omega \pa_t m(n-m)^- dx \\
  &= \int_\Omega(n-m)\na V\cdot\na(n-m)^- dx
  + m\int_\Omega\na V\cdot\na(n-m)^- dx \\
  &= \frac12\int_\Omega\na V\cdot\na[(n-m)^-]^2 dx
  + m\int_\Omega\na V\cdot\na(n-m)^- dx \\
  &= -\frac{1}{2\lambda^2}\int_\Omega(n-p-D+A)[(n-m)^-]^2 dx
  - \frac{1}{\lambda^2}\int_\Omega(n-p-D+A)(n-m)^- dx \\
  &\le \frac{M}{\lambda^2}\int_\Omega[(n-m)^-]^2 dx
  + \frac{1}{\lambda^2}(M+\|A\|_{L^\infty(\Omega)})
  \int_\Omega(n-m)^- dx,
\end{align*}
where we used the Poisson equation \eqref{1.V} and the upper bound $n,p,D\le M=M(T)$. Since $\pa_t m = -\mu m$, the last integral on the right-hand side can be absorbed by the left-hand side after choosing $\mu\ge 2\lambda^{-2}(M+\|A\|_{L^\infty(\Omega)})$. We conclude from Gronwall's lemma that $(n-m)^-(t)=0$ and hence $n(t)\ge m$ in $\Omega$, $t\le T$. The choice $c=m_0\exp(-\mu T)>0$ finishes the proof.


\section{Conclusions}\label{sec.con}

We have presented in this paper a coupled semiconductor device and electric network model, consisting of three-species drift-diffusion equations for the densities of the charged carriers and the Poisson equation for the electric potential as well as differential-algebraic equations for the node potentials, currents through inductors, and currents through voltage sources. We have proved the existence of global weak solutions, the free energy inequality, and the strict positivity of the charge carriers under suitable assumptions.

Our results can be extended in different directions. Our proofs work as well for an arbitrary number of species with a proper specification of the boundary conditions. The right-hand side of the free energy inequality in Theorem \ref{thm.ex} depends on the end time $T$. This can be improved as in \cite[Lemma 2.3]{JJZ23} if the boundary data are in thermal equilibrium (zero applied bias). The proof of the boundedness and strict positivity of the densities uniform in time is more delicate, since this clearly requires conditions on the boundary data. If the densities are uniformly bounded from below, we obtain an $L^2(\Omega)$ bound for the fluxes uniform in time. Another direction is the long-time asymptotics. It is shown in \cite[Theorem 3]{AlRo21} that the solution to the coupled system converges to the thermal equilibrium state as $t\to\infty$ under some conditions. A possible extension could be the proof of an exponential convergence rate. Finally, since the drift-diffusion model shows some features of memristor devices like pinched hysteresis loops \cite[Sec.~6]{JJZ23}, the coupled device--circuit system may used in neuromorphic computing to describe the dynamics in artificial neural networks.


\end{document}